\documentclass[12pt,reqno]{amsart}
\usepackage{amssymb}
\usepackage{amsmath}
\usepackage{amsthm}
\usepackage[left=3cm,top=3cm,right=3cm,bottom=3cm]{geometry}
\usepackage{hyperref}
\usepackage[usenames]{color}
\usepackage{enumerate}
\usepackage[normalem]{ulem} 
\usepackage{graphicx}

\usepackage{etoolbox}
\makeatletter
\patchcmd{\@maketitle}
  {\ifx\@empty\@dedicatory}
  {\ifx\@empty\@date \else {\vskip3ex \centering\footnotesize\@date\par\vskip1ex}\fi
   \ifx\@empty\@dedicatory}
  {}{}
\patchcmd{\@adminfootnotes}
  {\ifx\@empty\@date\else \@footnotetext{\@setdate}\fi}
  {}{}{}
\makeatother

\begin{document}

\newtheorem{theorem}{Theorem}[section]
\newtheorem{lemma}[theorem]{Lemma}
\newtheorem{definition}[theorem]{Definition}
\newtheorem{conjecture}[theorem]{Conjecture}
\newtheorem{proposition}[theorem]{Proposition}
\newtheorem{claim}[theorem]{Claim}
\newtheorem{fact}[theorem]{Fact}
\newtheorem{corollary}[theorem]{Corollary}
\newtheorem{observation}[theorem]{Observation}
\newtheorem{problem}[theorem]{Open Problem}

\theoremstyle{remark}
\newtheorem{rem}[theorem]{Remark}

\newcommand{\weighted}[1]{\mbox{2-WEIGHTED}^{{(#1)}}}
\newcommand{\Whp}{{\textit{W.h.p. }}}
\newcommand{\whp}{{\textit{w.h.p. }}}
\newcommand{\bin}{\textrm{Bin}}
\newcommand{\po}{\textrm{Po}}
\newcommand{\codeg}{\textrm{codeg}}
\newcommand{\E}{\mathrm{E}}
\newcommand{\V}{\mathrm{Var}}
\newcommand{\RT}{\textup{\textbf{RT}}}
\newcommand{\f}{\textup{\textbf{f}}}
\newcommand{\z}{\textrm{\textbf{z}}}
\newcommand{\ex}{\textrm{\textbf{ex}}}
\newcommand{\sqbs}[1]{\left[ #1 \right]}
\newcommand{\of}[1]{\left( #1 \right)}
\newcommand{\bfrac}[2]{\of{\frac{#1}{#2}}}
\renewcommand{\l}{\ell}
\newcommand{\sm}{\setminus}
\newcommand{\mc}[1]{\mathcal{#1}}
\newcommand{\h}[1]{\mathbb{H}^{({#1})}}
\newcommand{\hh}{\mathbb{H}}
\newcommand{\g}{\mathbb{G}}
\newcommand{\dist}{\mathrm{dist}}
\renewcommand{\l}{\ell}

\title[]{On the number of alternating paths in bipartite complete graphs}

\author{Patrick Bennett, Andrzej Dudek, Elliot Laforge}
\address{Department of Mathematics, Western Michigan University, Kalamazoo, MI}
\email{\tt \{patrick.bennett,\;andrzej.dudek,\;elliot.m.laforge\}@wmich.edu}
\thanks{The second author was supported in part by Simons Foundation Grant \#244712 and by the National Security Agency under Grant Number H98230-15-1-0172. The United States Government is authorized to reproduce and distribute reprints notwithstanding any copyright notation hereon. }

\begin{abstract}
Let $C \subseteq [r]^m$ be a code such that any two words of $C$ have Hamming distance at least $t$. It is not difficult to see that determining 
a code $C$ with the maximum number of words  is equivalent to finding the largest $n$ such that there is an $r$-edge-coloring of $K_{m, n}$ with the property that any pair of vertices in the class of size~$n$ has at least $t$ alternating paths (with adjacent edges having different colors) of length~$2$. 
In this paper we consider a more general problem from a slightly different direction. We are interested in finding maximum $t$ such that there is an $r$-edge-coloring of $K_{m,n}$ such that any pair of vertices in class of size~$n$ is connected by $t$ internally disjoint and alternating paths of length~$2k$. We also study a related problem in which we drop the assumption that paths are internally disjoint. Finally, we introduce a new concept, which we call alternating connectivity. Our proofs make use of random colorings combined with some integer programs. 
\end{abstract}

\date{\today}

\maketitle

\section{Introduction}

In this paper we study alternating paths in bipartite graphs. A path is \emph{alternating} if its adjacent edges have different colors. The notion of alternating paths was originally introduced by Bollob\'as and Erd\H{o}s~\cite{BE}, where the authors studied under which conditions an $r$-edge-colored complete graph contains an alternating Hamiltonian cycle. There is a broad literature on this subject for graphs, see, e.g.,~\cite{AFR,AG, CD, S}, and also for hypergraphs~\cite{DF, DFR}. Several other results, mainly of algorithmic nature, are also known, see,  e.g., a survey paper~\cite{JG}.

The motivation of this paper comes from coding theory. 
Recall that the \emph{Hamming distance} between vectors $\mathbf{x}$ and $\mathbf{y}$ in $[r]^m$ is defined to be the number of positions in which they differ, i.e.,
\[
\dist(\mathbf{x}, \mathbf{y}) = \big{|} \{ i: 1\le i\le m,\ \mathbf{x}(i)\neq \mathbf{y}(i)\} \big{|}.
\]
Let $\alpha_r(m, t)$ be the maximum size of a code $C \subseteq [r]^m$ such that any two elements of $C$ have Hamming distance at least $t$. 
We refer the reader to~\cite{P} for more details concerning coding theory.

Let $K_{m,n}$ be a complete bipartite graph on vertex set $[m]\cup [n]$.
Suppose that $c: E(K_{m,n}) \rightarrow [r]$ is an $r$-edge-coloring of $K_{m,n}$ with the property that every pair of vertices in~$[n]$ is connected by at least~$t$ alternating paths of length 2 (with 3 vertices).  This edge coloring can be represented as a collection of $n$ vectors of length $m$ with entries in $[r]$ in the sense that for a vertex $v \in [n]$ we define the vector $\mathbf{c}_v$ by $\mathbf{c}_v(u) = c(\{v,u\})$ for $u \in [m]$.  Hence the set of vectors $C = \{\mathbf{c}_v : v\in [n] \}$ completely encodes the edge coloring since every edge of $K_{m,n}$ will belong to exactly one of these vectors, and by looking at the right one we can determine the color assigned to the edge. 

Now notice that the number of alternating paths of length 2 between $v,w \in [n]$ is exactly the Hamming distance $\dist(\mathbf{c}_v,\mathbf{c}_w)$.  This is because a path of length 2 is of the form $v-u-w$ which is alternating if and only if $c(\{v,u\}) \neq c(\{u,w\})$ or equivalently $\mathbf{c}_v(u) \neq \mathbf{c}_w(u)$. Since $c$ has the property that every pair of vertices in the same partite set is connected by at least $t$ alternating paths of length 2 it follows that $C$ has minimum Hamming distance $t$,  thus we must have that $|C| = n \le \alpha_r(m,t)$.  
Consequently, determining $\alpha_r(m, t)$ is equivalent to finding the largest $n$ such that there is an $r$-coloring of the edge set of $K_{m, n}$ with the property that any pair of vertices in~$[n]$ has at least $t$ alternating paths of length $2$ connecting them. Clearly all such paths are internally disjoint.

In this paper, we approach the problem from a slightly different direction: instead of fixing the alphabet, word length, and $t$ and asking for the largest possible code, we fix the alphabet, word length, and code size and ask for the largest possible $t$. We also consider longer paths. 
Let $\kappa_{r,2k}(m,n)$ be maximum $t$ such that there is an $r$-coloring of the edges of $K_{m,n}$ such that any pair of vertices in class of size~$n$ is connected by $t$ internally disjoint and alternating paths of length~$2k$. As noted above, $\kappa_{r,2}(m,n)$ is related to coding theory. However we study $\kappa_{r,2k}(m,n)$ for general $k$. In terms of coding theory, each such path is an alternating sequence of codewords and indices, such that each pair of consecutive codewords use different letters in the index that connects them in the sequence.
We will show that for $r\ge2$ and $n\ge m\gg \log n$
\[
\kappa_{r,2}(m,n) \sim \left( 1-\frac{1}{r}\right) m,
\]
and for any $k\ge 2$ and $n\ge m\gg 1$
\[
\kappa_{r,2k}(m,n) \sim  \frac{m}{k}.
\]
These results are essentially best possible since for $m < \log_r n$ we have $\kappa_{r,2}(m,n) = 0$. Indeed, if $m < \log_r n$, then $n > r^m$ and so there must be two vertices in the class of size~$n$ that have the same vectors of colors, and so they have no alternating path of length~2 connecting them. 

We will also consider a related problem in which we drop the requirement that paths are internally disjoint. Let $\lambda_{\ell}(m,n)$ be the maximum $t$ such that there is a $2$-coloring of the edges of $K_{m,n}$ such that any pair of vertices is connected by $t$ alternating paths of length~$\ell$. If $\ell$ is even, then we consider only pairs in the partition class of size~$n$. Determining $\lambda_{\ell}(m,n)$ seems to be more difficult. In particular, we will show that 
\[
\lambda_{3}(m,n) \sim  mn/4
\quad\text{ and }\quad
\lambda_{4}(m,n) \sim  m^2n/8.
\]

In doing so we determine an optimal upper bound on the number of alternating paths of length 3 and 4 in $K_{m,n}$ under any 2-edge-coloring. 

Motivated by studying $\kappa_{r,\ell}(m,n)$, in the last section, we propose a new concept, which we call \emph{alternating connectivity} and define as maximum $t$ such that there is an $r$-edge-coloring of $G$ such that any pair of vertices is connected by $t$ internally disjoint and alternating paths of length~$\ell$. We will discuss it briefly and show that for complete graphs
\[
\kappa_{r,2}(K_n)\sim (1-1/r) n
\quad\text{ and }\quad
\kappa_{r,\ell}(K_n)\sim n/(\ell-1)
\]
for any $r\ge 2$ and $\ell\ge 3$.

Throughout this paper all asymptotics are taken in~$n$. For simplicity, we do not round numbers that are supposed to be integers either up or down; this is justified since these rounding errors are negligible to the asymptomatic calculations we will make. All logarithms are natural unless written explicitly otherwise.

\section{Paths of length 2} \label{sec:two}

In this section we only consider paths of length~2. 
As it was already mention in the introduction here instead of fixing the alphabet $[r]$, word length $m$, and $t$ and asking for the largest possible code in $[r]^m$ with minimum Hamming distance~$t$, we fix the alphabet, word length, and code size and ask for the largest possible $t$.

\begin{theorem}\label{thm:two}
Let $r\ge 2$ and $n\ge m\gg \log n$. Then,
\[
\kappa_{r,2}(m,n) \sim \left( 1-\frac{1}{r}\right) m. 
\]
\end{theorem}
\noindent
This is essentially the best possible since, as it was observed in the introduction, if $m<\log_r(n)$, then $\kappa_{r,2}(m,n)=0$. The proof of this theorem is an immediate consequence of Lemmas~\ref{lem:ub} and~\ref{lem:lb_for_2}.

\begin{lemma}\label{lem:ub}
Let $r\ge 2$ and $n \ge m$ be positive integers. Then, 
\begin{equation}\label{eq:lemma1:1}
\kappa_{r,2}(m,n) \le \left( 1-\frac{1}{r}\right) \left(1+\frac{1}{n-1}\right) m \sim \left( 1-\frac{1}{r}\right)m.
\end{equation}
\end{lemma}

\begin{proof}
Let $K_{m,n}$ be an $r$-edge-colored bipartite graph on $[m]\cup [n]$. For a vertex $v\in [m]$ let $\deg_i(v)$ 
denote the number of edges adjacent to $v$ which are colored by $i$.
Note that the total number of alternating paths of length 2 with the middle vertex $v$ is $\sum_{1\le i<j\le r} \deg_i(v) \deg_j(v)$,
which is
\begin{align*}
\sum_{1\le i<j\le r} \deg_i(v) \deg_j(v) &= \frac{1}{2} \left(\left( \sum_{i=1}^r \deg_i(v)  \right)^2 -  \sum_{i=1}^r \deg_i(v)^2 \right)\\
&= \frac{1}{2} \left(n^2 -  \sum_{i=1}^r \deg_i(v)^2 \right)
\le \frac{1}{2} \left(n^2 -  \frac{n^2}{r} \right) = \frac{r-1}{2r} n^2,
\end{align*}
where the last inequality follows from the Cauchy-Schwarz inequality.
Thus,
\[
\binom{n}{2} \kappa_{r,2}(m,n) \le m\cdot \frac{r-1}{2r} n^2,
\]
which is equivalent to~\eqref{eq:lemma1:1}. 
\end{proof}
%

\begin{lemma}\label{lem:lb_for_2}
Let $r\ge 2$ and  $n\ge m\gg \log n$. Then,
\[
\kappa_{r,2}(m,n) \ge \left( 1-\frac{1}{r}\right) m - o(m).
\]
\end{lemma}

\begin{proof}
Let $[m]\cup [n]$ be the set of vertices of $K=K_{m,n}$. To each edge in $K$ we assign a color from $\{1,\dots,r\}$ uniformly at random.
For $u,v\in [n]$, let $X_{u,v}$ be the random variable that counts the number of alternating paths between $u$ and $v$ of lengths~2. Clearly, $X_{u,v} \sim \bin(m,(r-1)/r)$. 

We will use the following form of Chernoff's bound (see, e.g., Theorem~2.1 in~\cite{JLR})
\begin{equation*}
\Pr(X \le \E(X)-t) \le \exp\left( -\frac{t^2}{2\E(X)} \right),
\end{equation*}
where $X$ is a random variable with binomial distribution.

Since $\E(X_{u,v}) = (r-1)m/r$ and $m\gg \log n$, we get for $t=\sqrt{5\E(X_{u,v})\log n}$ that $\E(X)-t = (r-1)m/r - o(m)$ and so
\[
\Pr\left(X_{u,v} \le \frac{r-1}{r}m - o(m)\right) \le \exp\left( -\frac{t^2}{2\E(X_{u,v})} \right) = \exp(-5(\log n)/2).
\]
Thus, the union bound taken over all $\binom{n}{2}\le \exp(2\log n)$ pairs of vertices in $[n]$ yields the statement.
\end{proof}

\section{Paths of arbitrary length} \label{thm:any_length}

Here we consider $\kappa$ for paths of any length. Quite surprisingly two colors already suffice to get an optimal result.

\begin{theorem}\label{thm:even}
Let $r\ge2$, $k\ge 2$ and $n\ge m\gg 1$. Then,
\[
\kappa_{r,2k}(m,n) \sim  \frac{m}{k}. 
\]
\end{theorem}
\noindent
The only case not covered by the above theorem is when $m$ is a constant. For example, in this regime is not difficult to see that
$\kappa_{2,4}(2,n)=0$ and $\kappa_{2,4}(3,n)=1$ for any $n$ sufficiently large. 

The proof of Theorem~\ref{thm:even} is based on the following lemma. Denote by $\mathbb{G}(m,m,1/2)$ the (truly) \emph{random bipartite graph} in which every possible edge between two partition classes of size~$m$ is chosen independently with probability~$1/2$. We say that an event $E_n$ occurs \emph{with high probability}, or \whp for brevity, if $\lim_{n\rightarrow\infty}\Pr(E_n)=1$.
\begin{lemma}\label{prop:matchings}
Let $m \gg 1$ and let $0<\alpha<1$. Let $\mathbb{G}(m,m,1/2)$ be a random bipartite graph on set $M_1\cup M_2$ with $|M_1|=|M_2|=m$.
Then, \whp for each $A\in \binom{M_1}{\alpha m}$ and $B\in \binom{M_2}{\alpha m}$ there exists a matching between $A$ and $B$ of size $\alpha m - o(m)$.
\end{lemma}

\begin{proof}
Fix $A\in \binom{M_1}{\alpha m}$ and $B\in \binom{M_2}{\alpha m}$ and let $G=\mathbb{G}(|A|,|B|,1/2)$. First we consider an auxiliary bipartite graph $H$ on $U\cup W$ such that $U=A\cup A'$, $W=B\cup B'$, $H[A\cup B] = G$, and $H[A'\cup W]$ and $H[U\cup B']$ are complete bipartite graphs. Furthermore, let $s=\log m=|A'|=|B'|$.
We show that $H$ has a perfect matching. It suffices to show that the Hall condition holds, i.e., 
\begin{equation}\label{eq:hallS}
\text{if } S\subseteq U \text{ and } |S| \le |U|/2, \text{ then } |N(S)| \ge |S|,
\end{equation}
and
\begin{equation}\label{eq:hallT}
\text{if } T\subseteq W \text{ and } |T| \le |W|/2, \text{ then } |N(T)| \ge |T|.
\end{equation}
If $|S| < s$, then since $N(S)\supseteq B'$, $|N(S)| \ge |B'|=s \ge |S|$. Therefore, we assume that $s\le |S| \le |U|/2$. Furthermore, we may assume that   $S\cap A' = \emptyset$. For otherwise, $N(S) = W$.
We will show that already for $|S|=s$, $|N(S)| \ge |W|/2=|U|/2$.

Suppose not, that is, $|N(S)| < (\alpha m + \log m)/2$. That means $|B\cap N(S)| < (\alpha m -\log m)/2$, $e(S,B\setminus N(S))=0$ and $|B\setminus N(S)|\ge (\alpha m+\log m)/2>\alpha m/2$.
Observe that the probability that there are sets $S\in A$ and $T\in B$ such that $|S|=s$ and $|T| = \alpha m/2$ and $e(S,T)=0$ is at most
\[
\binom{\alpha m}{s} \binom{\alpha m}{\alpha m/2} 2^{- s \cdot \alpha m/2} 
\le 2^{\alpha m}\cdot 2^{\alpha m} \cdot 2^{- s \cdot \alpha m/2} = 2^{2\alpha m- s \cdot \alpha m/2}.
\]
Thus, with probability at most $2^{2\alpha m- s \cdot \alpha m/2}$ graph $H$ violates \eqref{eq:hallS}, and similarly~\eqref{eq:hallT}.
In other words, with probability at least $1-2\cdot 2^{2\alpha m- s \cdot \alpha m/2}$ graph $H$ has a perfect matching, and consequently, there is a matching of size $\alpha m - s$ between $A$ and $B$.

Finally, by taking the union bound over all  $A\in \binom{M_1}{\alpha m}$ and $B\in \binom{M_2}{\alpha m}$ we get that the probability that there exist $A$ and $B$ such that between $A$ and $B$ there is no matching of size $\alpha m - s$ is at most
\[
\binom{m}{\alpha m} \binom{m}{\alpha m} 2^{2\alpha m- s \cdot \alpha m/2+1} \le 2^m \cdot 2^m  \cdot 2^{2\alpha m- s \cdot \alpha m/2+1} = o(1),
\]
since $s=\log m$. Also, clearly, we get that $\alpha m - s = \alpha m - o(m)$. Thus, \whp for each $A$ and $B$ there is a matching between $A$ and $B$ of size $\alpha m - o(m)$.
\end{proof}

\begin{proof}[Proof of Theorem~\ref{thm:even}]

First observe that the upper bound is trivial since $(1+o(1))m/k$ paths of length $2k$ saturate all vertices in $M$.

Let $M\cup N$ be the bipartition of $K=K_{m,n}$ such that $|M|=m$ and $|N|=n$. Furthermore, let $N = N' \cup N''$, where $|N'|=m$ and $|N''|=n-m$. To each edge in $K$ induced by $M\cup N'$ we assign a color from $\{blue,red\}$ uniformly at random. To all other edges (between $M$ and $N''$) we assign colors in such a way that for each $v\in N''$, $\deg_B(v) = \deg_R(v)$ and for each $w\in M$ and $u,v\in N''$, $\{u,w\}$ and $\{v,w\}$ have the same color (in other words the color vector of each $v\in N''$ is the same). Observe that both the red graph induced on $M\cup N'$ and the blue one can be viewed as $\mathbb{G}(m,m,1/2)$. So actually our $r$-edge-coloring uses only 2 colors.

First we consider $u$ and $v$ in $N''$. Let $X_B(u)$ and $X_R(v)$ be two disjoint subsets of $M$ such that all edges between $u$ and $X_B(u)$ are blue and all edges between $v$ and $X_R(v)$ are red  and $|X_B(u)|=|X_R(v)|=m/k$.
Now we choose subsets $X_1,\dots,X_{k-2}\subseteq M$ and $Y_1,Y_2,\dots,Y_{k-1}\subseteq N'$ which are pairwise disjoint and also disjoint with $X_B(u)$ and $X_R(v)$, and $|X_i|=|Y_j| = m/k$. By Lemma~\ref{prop:matchings} there is a red matching between $X_B(u)$ and $Y_1$, a blue matching between $Y_1$ and $X_1$, etc., each of size $m/k-o(m)$. This yields $m/k-o(m)$ alternating and internally disjoint paths between $u$ and~$v$. 

Now consider $u$ and $v$ in $N'$  and define
\[
N_{xy}(u,v) = \{w\in M: \{w,u\} \text{ has color } x \text{ and } \{w,v\} \text{ has color } y\}. 
\]
Chernoff's bound implies that $|N_{RR}(u,v)| \sim |N_{RB}(u,v)| \sim |N_{BR}(u,v)|\sim |N_{BB}(u,v)| \sim m/4$ for all $u, v \in N'$.
Let $X_B(u)\subseteq N_{BR}(u,v)$, $X_R(v)\subseteq N_{RR}(u,v)$ such that $|X_B(u)| = |X_R(v)| = m/(2k)+o(m)$. In other words,  all edges between $u$ and $X_B(u)$ are blue and all edges between $v$ and $X_R(v)$ are red. Now we choose disjoint subsets $X_1,\dots,X_{k-2}\subseteq N_{BR}(u,v) \cup N_{RR}(u,v) \setminus (X_B(u)\cup X_R(v))$ and $Y_1,\dots,Y_{k-1}\subseteq N'$  such that $|X_i|=|Y_j| = m/(2k)+o(m)$. By Lemma~\ref{prop:matchings} there is a red matching between $X_B(u)$ and $Y_1$, a blue matching between $Y_1$ and $X_1$, etc., each of size $m/(2k)-o(m)$. This yields $m/(2k)-o(m)$ alternating and internally disjoint paths between $u$ and~$v$. To find the remaining $m/(2k)-o(m)$ paths we choose $X'_R(u)\subseteq N_{RB}(u,v)$, $X'_B(v)\subseteq N_{BB}(u,v)$, 
$X'_1,\dots,X'_{k-2}\subseteq N_{RB}(u,v) \cup N_{BB}(u,v) \setminus (X'_R(u)\cup X'_B(v))$, and 
$Y'_1,\dots,Y'_{k-1}\subseteq N'\setminus(Y_1\cup\dots\cup Y_{k-1})$.

Finally observe that the case $u\in N'$ and $v\in N''$ is very similar to the latter.
\end{proof}

Note that the above proof would not work for all $m \gg 1$ if we used the simpler strategy of just coloring all edges randomly. In particular, to get the concentration of degrees we would need $m$ to be at least on the order of $\log n$.

In a similar manner one can define $\kappa_{r,2k+1}(m,n)$. Now, clearly the endpoints are in different partition classes. This case is not interesting since one can easily see that $\kappa_{r,2k+1}(m,n)\sim m/k$. For the lower bound color by red one matching saturating all vertices in the class of size $m$ and all other edges blue. As in the previous proof $m/k$ is best possible.

\section{Not necessary internally disjoint paths}

In this section we are not assuming anymore that paths must be internally disjoint. Furthermore, we will only consider 2-colorings. Recall that $\lambda_{\ell}(m,n)$ denotes the maximum $t$ such that there is a $2$-edge-coloring of $K_{m,n}$ such that any pair of vertices is connected by $t$ alternating paths of length~$\ell$. If $\ell$ is even, then we consider only pairs in partition class of size~$n$.

In general determining $\lambda$ seems to be more difficult than $\kappa$ and as we will see the corresponding results for $\lambda$ significantly differ from those for $\kappa$.

We start with a lower bound.

\begin{theorem}\label{lem:gen:lb}
Let $k \ge 1$ be an integer.
\begin{enumerate}[(i)]
\item\label{lem:gen:lb:even} If $n\ge m\gg \log n$, then
\begin{equation}\label{eq:gen:lb:even}
\lambda_{2k}(m,n) \ge (1+o(1)) m^k n^{k-1}  /2^{2k-1}.
\end{equation}
\item\label{lem:gen:lb:odd} If $n\ge m = \Omega(1)$, then
\begin{equation}\label{eq:gen:lb2:odd}
\lambda_{2k+1}(m,n) \ge (1+o(1)) (mn/4)^k.
\end{equation}
\end{enumerate}
\end{theorem}

\begin{proof}
For \eqref{lem:gen:lb:odd} it suffices to consider the following coloring. Let $M=M_1\cup M_2$ and $N=N_1\cup N_2$, $|M_i|=m/2$ and $|N_i|=n/2$,
be partition classes of $K_{m,n}$. Color the edges between $M_i$ and $N_i$ red and blue otherwise. It is easy to see that this coloring yields $\lambda_{2k+1}(m,n) \ge (1+o(1))(m/2)^k (n/2)^k$.

Now we show \eqref{lem:gen:lb:even}. We use the concentration of degrees and codegrees. We fix vertices $u$ and $v$ in $N$. We then choose some $x_1 \in M$ such that $\{u, x_1\}$ is red, some $x_2\in N$ such that $\{x_1, x_2\}$ is blue, and so on until we reach $x_{2k-2} \in N$. So far there are asymptotically $\frac m2$ choices for each $x_i$ if $i$ is odd, and $\frac n2$ choices if $i$ is even. Now we have to choose $x_{2k-1}$ such that the edge $\{x_{2k-2}, x_{2k-1}\}$ is red and $\{x_{2k-1}, v\}$ is blue. There are asymptotically $\frac m4$ choices for $x_{2k-1}$.  This gives $m^kn^{k-1} / 2^{2k}$ choices for paths with a red edge adjacent to $u$. Similarly, we estimate the number of paths with a blue edge adjacent to $u$. 
\end{proof}

It is not clear whether Theorem~\ref{lem:gen:lb} is optimal in general. Here we managed to show tight upper bounds on $\lambda_{3}(m,n)$ and $\lambda_4(m,n)$. Clearly also $\lambda_{2}(m,n) = \kappa_{2,2}(m,n)\sim m/2$.
\begin{theorem}
Let $n\ge m\ge 1$. Then,
\[
\lambda_{3}(m,n) \le  mn/4
\quad\text{ and }\quad
\lambda_{4}(m,n) \le (1+o(1)) m^2n/8.
\]
\end{theorem}
This theorem will immediately follow from the following result.
\begin{lemma}\label{lem:3_4} 
Let the edges of $K_{m,n}$ be 2-colored. Then, the number of all alternating paths of length 3 is at most $m^2n^2/4$ and the number of all alternating paths of length~4 with two endpoints in the class of size~$n$ is at most $m^2n^3/16$.
\end{lemma}
\noindent
Indeed, this lemma implies that
\[
\lambda_{3}(m,n) \le  (m^2n^2/4) \big{/} mn =  mn/4,
\]
and
\[
\lambda_{4}(m,n) \le (m^2n^3/16) \big{/} \binom{n}{2} = (1+o(1)) m^2n/8.
\]

\begin{proof}[Proof of Lemma~\ref{lem:3_4}] 
Let $K=K_{m,n}$ be a blue-red edge-colored bipartite graph on $[m]\cup [n]$.

First we count the total number of alternating paths of length 3. Each such path has two vertices $u, v \in [m]$ and two vertices $x_1, x_2 \in [n]$. We assume that $u$ and $v$ are fixed and $u<v$. First we determine the number of red-blue-red paths containing $u$ and $v$. Here we have two possibilities: either $u-x_1-v-x_2$ is red-blue-red or $x_1-u-x_2-v$ is red-blue-red. This yields 
\[
\codeg_{RB}(u,v)\deg_R(v) + \deg_R(u) \codeg_{BR}(u,v)
\]
number of choices, where $\codeg_{xy}(u,v) = |N_{xy}(u,v)|$. Similarly, we see that the number of blue-red-blue paths containing $u$ and $v$ is
\[
\codeg_{BR}(u,v)\deg_B(v) + \deg_B(u) \codeg_{RB}(u,v).
\]
Thus, the number of alternating path of length~3 is never bigger than the solution of the following integer program:
\begin{center}
\begin{tabular}{l l l}
& \texttt{Maximize}    & $\sum_{1\le u<v\le m}  x_{BR}(u,v)[x_R(u)+x_B(v)] + x_{RB}(u,v)[x_B(u)+x_R(v)] $ \\[8pt]
     & \texttt{subject to:} & \textit{The solution is graphical.}\\[5pt]
\end{tabular}
\end{center}
We say that a solution is \emph{graphical} if there exists a 2-edge-coloring of $K_{m,n}$ which realizes all color degrees and codegrees corresponding to the variables of the program.

We will find an upper bound on the solution of this program. First we show that 
\begin{equation}\label{eq:codegrees}
x_{BR}(u,v) - x_{RB}(u,v) = x_B(u) - x_B(v).
\end{equation}
Let $X_B(u)$ and $X_B(v)$ be the set of vertices that are connected to $u$ and $v$ by blue edges, respectively. Furthermore, let $X_{BR}(u,v)$ be the set of vertices $w$ such that $\{u,w\}$ is blue and $\{v,w\}$ is red. Similarly, we define $X_{RB}(u,v)$. Now since the solution is graphical
\[
X_{BR}(u,v) = X_B(u) \setminus (X_B(u)\cap X_B(v))
\ \text{ and } \
X_{RB}(u,v) = X_B(v) \setminus (X_B(u)\cap X_B(v)),
\]
we get
\begin{align*}
x_{BR}(u,v) - x_{RB}&(u,v) = |X_{BR}(u,v)| - |X_{RB}(u,v)|\\
&= |X_B(u) \setminus (X_B(u)\cap X_B(v))| - |X_B(v) \setminus (X_B(u)\cap X_B(v))|\\
&= \left(|X_B(u)| - |(X_B(u)\cap X_B(v))|\right) - \left(|X_B(v)| -  |(X_B(u)\cap X_B(v))|\right) \\
&=|X_B(u)| -  |X_B(v)| = x_B(u) - x_B(v)
\end{align*}
proving~\eqref{eq:codegrees}.

Set $c(u,v) = x_{RB}(u,v) + x_{BR}(u,v)$. Due to~\eqref{eq:codegrees}
\begin{equation}\label{eq:xbr}
x_{BR}(u,v) = \frac{c(u,v) + x_B(u) - x_B(v)}{2}
\end{equation}
and
\begin{equation}\label{eq:xrb}
x_{RB}(u,v) = \frac{c(u,v) - x_B(u) + x_B(v)}{2}.
\end{equation}
Furthermore, since 
\[
x_R(u) = n-x_B(u)\quad \text{ and }\quad x_R(v) = n-x_B(v),
\]
we get
\begin{align*}
x_{BR}(u,v)[x_R(u)&+x_B(v)] + x_{RB}(u,v)[x_B(u)+x_R(v)]\\
&= \frac{c(u,v)+[x_B(u)-x_B(v)]}{2} (n - [x_B(u)-x_B(v)])\\
&\qquad\qquad+\frac{c(u,v)-[x_B(u)-x_B(v)]}{2} (n + [x_B(u)-x_B(v)])\\
&= c(u,v)n - [x_B(u)-x_B(v)]^2 \le c(u,v)n.
\end{align*}
Thus,
\begin{align}\label{proof:3:ub}
\sum_{1\le u<v\le m} x_{BR}(u,v)[x_R(u)+x_B(v)] + x_{RB}(u,v)[x_B(u)&+x_R(v)] \notag \\
& \le n \sum_{1\le u<v\le m} c(u,v). 
\end{align}
Note that for any 2-coloring of the edges of $K$
\[
\sum_{1\le u<v\le m} \codeg_{RB}(u,v) + \codeg_{BR}(u,v) = \sum_{w\in N}\deg_R(w)\deg_B(w),
\]
and since
\[
\sum_{w\in [n]}\deg_R(w)\deg_B(w) \le \sum_{w\in [n]}\left( \frac{\deg_R(w)+\deg_B(w)}{2} \right)^2 = n\cdot \left(\frac{m}{2}\right)^2,
\]
we obtain
\begin{equation}\label{eq:cuv}
\sum_{1\le u<v\le m} c(u,v) = \sum_{1\le u<v\le m}x_{RB}(u,v) + x_{BR}(u,v) \le m^2n/4.
\end{equation}
Thus,  \eqref{proof:3:ub} and \eqref{eq:cuv} imply that $\lambda_3(m,n) \le  m^2n^2/4$.

Now we count the total number of alternating paths of length 4 with both endpoints in~$N$. Each such path is of the form $x_1-u-x_2-v-x_3$, where $u,v\in [m]$ and $x_1,x_2,x_3\in [n]$. Similarly as in the previous case we fix $u,v\in [m]$ with $u<v$ and count the number of paths going through $u,v$. Thus, the number of red-blue-red-blue paths is at most
\[
\sum_{1\le u<v\le m} \deg_R(u)\codeg_{BR}(u,v)\deg_B(v), 
\] 
(we do not assume here that $x_1$, $x_2$ and $x_3$ are different).
Similarly the number of blue-red-blue-red paths is at most
\[
\sum_{1\le u<v\le m} \deg_B(u)\codeg_{RB}(u,v)\deg_R(v) .
\]
Thus, the number of alternating path of length 4 is bounded from above by the solution of the following integer program:
\begin{center}
\begin{tabular}{l l l}
& \texttt{Maximize}    & $\sum_{1\le u<v\le m}  x_R(u)x_{BR}(u,v)x_B(v) + x_B(u)x_{RB}(u,v)x_R(v)  $ \\[8pt]
     & \texttt{subject to:} & \textit{The solution is graphical.}\\[5pt]
\end{tabular}
\end{center}
As before we set $c(u,v) = x_{RB}(u,v) + x_{BR}(u,v)$ and by~\eqref{eq:xbr} and~\eqref{eq:xrb} we get
\begin{align*}
 x_R(u)x_{BR}(u,v)x_B(v)& + x_B(u)x_{RB}(u,v)x_R(v) =\\
 &= \frac{c(u,v) + x_B(u) - x_B(v)}{2} \cdot [n-x_B(u)] x_B(v)\\ 
 &\qquad\qquad+ \frac{c(u,v) - x_B(u) + x_B(v)}{2} \cdot [n-x_B(v)] x_B(u)\\
 &= \frac{c(u,v)}{2} \big{(} x_B(u)[n-x_B(u)]  + x_B(v)[n-x_B(v)] \big{)}\\
 &\qquad\qquad+\frac{1}{2}[x_B(u) - x_B(v)]^2 [c(u,v)-n].
\end{align*}
The last equality follows just from simple algebra operations. Since $c(u,v) \le n$, the  second term is at most 0. Furthermore, clearly
\[
x_B(u)[n-x_B(u)] \le n^2/4
\quad\text{ and }\quad
x_B(v)[n-x_B(v)] \le n^2/4.
\] 
Thus,
\[
 x_R(u)x_{BR}(u,v)x_B(v) + x_B(u)x_{RB}(u,v)x_R(v) \le c(u,v)n^2/4.
\]
Consequently, by~\eqref{eq:cuv}
 \begin{align*}
 \sum_{1\le u<v\le m}  x_R(u)x_{BR}(u,v)x_B(v) &+ x_B(u)x_{RB}(u,v)x_R(v)  \\
& \le \frac{n^2}{4}  \sum_{1\le u<v\le m} c(u,v) \le \frac{n^2}{4} \cdot \frac{m^2n}{4} = \frac{m^2n^3}{16}
 \end{align*}
 finishing the proof for path of length~4.
 \end{proof}

%

Lemma~\ref{lem:3_4} establishes the maximum number of alternating path of length 3 and 4. It is not difficult to see that in theory the approach taken in the proof of this lemma can be extended to counting alternating paths of any length. For example, binding from above the number of paths of length 5 corresponds to the following integer program.
\begin{center}
\begin{tabular}{l l l}
& \texttt{Maximize}    & $\displaystyle{\sum_{1\le u<w<v\le m} f(u,v,w)}$ \\[20pt]
     & \texttt{subject to:} & \textit{The solution is graphical}\\[5pt]
\end{tabular}
\end{center}
where
\begin{align*}
f(u,w,v) &= x_{BR}(u,w)x_{BR}(w,v)[x_R(u)+x_B(v)]\\
&\qquad+ x_{RB}(u,w)x_{RB}(w,v)[x_B(u)+x_R(v)]\\
&\qquad\qquad+ x_{RB}(u,w)x_{BR}(u,v)[x_R(w)+x_B(v)]\\
&\qquad\qquad\qquad+ x_{BR}(u,w)x_{RB}(u,v)[x_B(w)+x_R(v)]\\
&\qquad\qquad\qquad\qquad+ x_{BR}(u,v)x_{RB}(w,v)[x_R(u)+x_B(w)]\\
&\qquad\qquad\qquad\qquad\qquad+ x_{RB}(u,v)x_{BR}(w,v)[x_B(u)+x_R(w)].
\end{align*}
Unfortunately, the objective function is more complicated and computations become more difficult and technical.

One can also show that the problem of counting alternating paths in $K=K_{m,n}$ can be reduced to finding directed paths in a certain digraph. Indeed, let a 2-coloring of edges of $K$ be given. Furthermore, assume that $M\cup N$ is the corresponding  bipartition. We build a bipartite digraph $D$ on $M\cup N$ as follows. For each red edge $\{u,v\}$ with $u\in N$ and $v\in M$ we define a directed edge $uv$ in $D$; otherwise, if $\{u,v\}$ is blue, then we add to $D$ directed edge $vu$. In other words, all red edges are oriented from $N$ to $M$ and blue ones from $M$ to $N$.
Clearly, each alternating path in $K$ corresponds to a (unique) directed path in $D$. Hence, we just reduced the problem of counting alternating paths in $K$ to counting directed paths in bipartite tournament. Unfortunately, in general the problem of counting directed paths does not seem to be easy (see, e.g., \cite{SS}). 

Finally let us mention that one can also study $\lambda$ for any number~$r$ of colors, denoted by $\lambda_{r,\ell}(m,n)$. Since $\lambda_{r,2}(m,n)=\kappa_{r,2}(m,n)$, Theorem~\ref{thm:two} implies that $\lambda_{r,2}(m,n)\sim (1-1/r)m$ provided that $n\ge m\gg \log n$. By assigning to the edges of $K_{m,n}$ colors from set $[r]$ uniformly at random one can see (like in the proof of Theorem~\ref{lem:gen:lb}) that $\lambda_{r,2k}(m,n) \ge (1+o(1)) m^k n^{k-1}  (1-1/r)^{2k-1}$ and $\lambda_{2k+1}(m,n) \ge (1+o(1)) (mn)^k (1-1/r)^{2k}$. The optimality of these bounds remains open.

\section{Remarks on alternating connectivity}

Motivated by studying internally disjoint and alternating paths in a complete bipartite graph we introduce alternating connectivity.
Let $\kappa_{r,\ell}(G)$ be \emph{alternating connectivity} of a graph~$G$ defined as maximum $t$ such that there is an $r$-edge-coloring of $G$ such that any pair of vertices is connected by $t$ internally disjoint and alternating paths of length~$\ell$. 

As in Sections~\ref{sec:two} and~\ref{thm:any_length} we show that for complete graphs the following holds.

\begin{theorem}
Let $n\gg 1$. Then, for any number of colors $r\ge 2$
\begin{equation}\label{eq:alter_con_i}
\kappa_{r,2}(K_n)\sim (1-1/r) n,
\end{equation}
and for $\ell\ge 3$
\begin{equation}\label{eq:alter_con_ii}
\kappa_{r,\ell}(K_n)\sim n/(\ell-1).
\end{equation}
\end{theorem}

\begin{proof}
For~\eqref{eq:alter_con_i}  observe that like in the proof of Lemma~\ref{lem:ub} the total number of alternating paths of length~2 in $K_n$ is at most (due to the Cauchy-Schwarz inequality)
\begin{align*}
\sum_{v\in V(K_n)}   \sum_{1\le i<j\le r} \deg_i(v) \deg_j(v) &= \sum_{v\in V(K_n)} \frac{1}{2} \left(\left( \sum_{i=1}^r \deg_i(v)  \right)^2 -  \sum_{i=1}^r \deg_i(v)^2 \right)\\
&= \sum_{v\in V(K_n)}\frac{1}{2} \left((n-1)^2 -  \sum_{i=1}^r \deg_i(v)^2 \right)
\le \frac{r-1}{2r} (n-1)^2n,
\end{align*}
and so 
\[
\kappa_{r,2}(K_n) \le \frac{r-1}{2r} (n-1)^2n \big{/} \binom{n}{2} = \left( 1-\frac{1}{r}\right) (n-1).
\]
On the other hand, assigning colors from $\{1,\dots,r\}$ uniformly at random to the edges of $K_n$ (as in the proof of Lemma~\ref{lem:lb_for_2}) yields $\kappa_{r,2}(K_n)\ge  (1-1/r) n - o(n)$, provided that $n\gg 1$. This finishes the proof of~\eqref{eq:alter_con_i}. 

For~\eqref{eq:alter_con_ii} we slightly modify proofs from Section~\ref{thm:any_length}. Clearly, the upper bound is trivial, since any collection of $n/(\ell-1)$ internally disjoint paths of length~$\ell$ must saturate all vertices. For the lower bound we will show that there exists a 2-coloring of the edges with the required property. First observe that one can easily generalize Lemma~\ref{prop:matchings} for a (truly) random graph $\mathbb{G}(n,1/2)$, where each edge is chosen independently with probability $1/2$.
\begin{lemma}\label{prop:matchings_r}
Let $n \gg 1$ and $0<\alpha<1/2$. Let $\mathbb{G}(n,1/2)$ be a random graph on set of vertices~$[n]$.
Then, \whp for each $A\in \binom{[n]}{\alpha n}$ and $B\in \binom{[n]}{\alpha n}$ with $A\cap B = \emptyset$ there exists a matching between $A$ and $B$ of size $\alpha n - o(n)$.
\end{lemma}

Now we are ready to finish the proof of~\eqref{eq:alter_con_ii}. Assign colors from $\{blue,red\}$ uniformly at random to the edges of $K_n$. 
Let $u,v\in V(K_n)$. 
By Chernoff's bound, 
$|N_{RR}(u,v)| \sim |N_{RB}(u,v)| \sim |N_{BR}(u,v)|\sim |N_{BB}(u,v)| \sim n/4$.

Now if $\ell$ is even, then we proceed like in the proof of Theorem~\ref{thm:even}.
Let $X_B(u)\subseteq N_{BR}(u,v)$, $X_R(v)\subseteq N_{RR}(u,v)$ such that $|X_B(u)| = |X_R(v)| = n/(2(\ell-1))-o(n)$. In other words,  all edges between $u$ and $X_B(u)$ are blue and all edges between $v$ and $X_R(v)$ are red. Now we choose disjoint subsets $X_1,\dots,X_{\ell-3}\subseteq N_{BR}(u,v) \cup N_{RR}(u,v) \setminus (X_B(u)\cup X_R(v))$ such that $|X_i|=n/(2(\ell-1))+o(n)$. 
By Lemma~\ref{prop:matchings_r} there is a red matching between $X_B(u)$ and $X_1$, a blue matching between $X_1$ and $X_2$, etc., each of size $n/(2(\ell-1))-o(n)$. This yields $n/(2(\ell-1))-o(n)$ alternating and internally disjoint paths between $u$ and~$v$. To find the remaining $n/(2(\ell-1))-o(n)$ paths we choose $X'_R(u)\subseteq N_{RB}(u,v)$, $X'_B(v)\subseteq N_{BB}(u,v)$, and $X'_1,\dots,X'_{k-2}\subseteq N_{RB}(u,v) \cup N_{BB}(u,v) \setminus (X'_R(u)\cup X'_B(v))$.

Finally, if $\ell$ is odd, then we choose $X_B(u)\subseteq N_{BR}(u,v)$, $X_B(v)\subseteq N_{BB}(u,v)$, and $X_1,\dots,X_{\ell-3}\subseteq N_{BR}(u,v) \cup N_{BB}(u,v) \setminus (X_B(u)\cup X_B(v))$, and for the remaining paths $X'_R(u)\subseteq N_{RB}(u,v)$, $X'_R(v)\subseteq N_{RR}(u,v)$, and $X'_1,\dots,X'_{\ell-3}\subseteq N_{RB}(u,v) \cup N_{RR}(u,v) \setminus (X'_R(u)\cup X'_R(v))$.
\end{proof}

It might be of some interest to investigate $\kappa_{r,\ell}(G)$ for an arbitrarily graph~$G$. For example, studying alternating connectivity of random graphs could lead to better understanding of this concept.

\providecommand{\bysame}{\leavevmode\hbox to3em{\hrulefill}\thinspace}
\providecommand{\MR}{\relax\ifhmode\unskip\space\fi MR }
\providecommand{\MRhref}[2]{%
  \href{http://www.ams.org/mathscinet-getitem?mr=#1}{#2}
}
\providecommand{\href}[2]{#2}

\end{document}